\documentclass[12points,reqno]{amsart}
\usepackage{amsfonts}
\usepackage{amssymb}
\usepackage{graphicx}
\usepackage{amsmath}

\newtheorem{theorem}{Theorem}
\newtheorem{lemma}{Lemma}
\newtheorem{prop}{Proposition}
\newtheorem{corollary}{Corollary}
\theoremstyle{definition}

\newtheorem{example}{Example}

\theoremstyle{remark}

\numberwithin{equation}{section}

\everymath{\displaystyle}

\begin{document}
\title{ Centers of Leavitt path algebras and their completions }

%\address{
%a. Department of Mathematics,
%King Abdulaziz University,
%Jeddah, SA\\
%\newline\email{analahmadi@kau.edu.sa; hhaalsalmi@kau.edu.sa}
%\newline
%b. Department of Mathematics\\
%Ohio University, Athens, USA \\
%\newline\email{jain@ohio.edu}
%\newline
%c. Department of Mathematics\\
%University of California, San Diego, USA\\
%\newline\email{ezelmano@math.ucsd.edu}
%\newline
%1. To whom correspondence should be addressed\\ E-mail: ezelmano@math.ucsd.edu.
%\newline
%Author Contributions: A. A., H. A., S. K. J., E. Z. designed research;
%performed research and wrote the paper. The authors declare no conflict of interest. }
%
%\author{Adel Alahmedi $^a$, Hamed Alsulami$^a$, S. K. Jain $^{a,b}$
%, Efim Zelmanov$^{a,c,1}$}

\author{Adel Alahmadi}
\address{Department of Mathematics, Faculty of Science, King Abdulaziz University, P.O.Box 80203, Jeddah, 21589, Saudi Arabia}
\email{analahmadi@kau.edu.sa}
\author{Hamed Alsulami}
\address{Department of Mathematics, Faculty of Science, King Abdulaziz University, P.O.Box 80203, Jeddah, 21589, Saudi Arabia}
\email{hhaalsalmi@kau.edu.sa}
%\author{Adel Alahmadi}
%\address{Department of Mathematics, King Abdulaziz University, P.O.Box 80203, Jeddah, 21589, Saudi Arabia}
%\email{analahmadi@kau.edu.sa}
%\author{Hamed Alsulami}
%\address{Department of Mathematics, King Abdulaziz University, P.O.Box 80203, Jeddah, 21589, Saudi Arabia}
%\email{hhaalsalmi@kau.edu.sa}
%
%\author{S. K. Jain}
%\address{Department of Mathematics, Ohio University, Athens, USA }
%\email{jain@ohio.edu}
%
%\author{Efim Zelmanov}
%\address{Department of Mathematics, University of California, San Diego, USA }
%\email{efim.zelmanov@gmail.com}


\maketitle

\begin{abstract}
 In [8, 9] M. G. Corrales Garcia, D. M. Barquero, C. Martin Gonzalez,
M. Siles Molina, J. F Solanilla Hernandez described the center of a Leavitt
path algebra and characterized it in terms of the underlying graph. We offer
a different characterization of the center. In particular, we prove that the
Boolean algebra of central idempotents \ of a Leavitt path algebra of a
finite graph is isomorphic to the Boolean algebra of finitary annihilator
hereditary subsets of the graph.
\end{abstract}

\maketitle

\section{Definitions and The Main Results}

A (directed) graph $\Gamma =(V,E,s,r)$ consists of two sets $V$ and $E$ that
are respectively called vertices and edges, and two maps $s,r:E\rightarrow V.$ The vertices $s(e),$ $r(e)$ are referred to as the source and the range of
the edge $e$ respectively. The graph is called row-finite if for all vertices $v\in V$, $card(s^{-1}(v))<\infty $. A vertex $v$ for which $%
s^{-1}(v)=\phi $ is called a sink.

A path $p=e_{1}\cdots e_{n}$ of length $l(p)=n$ in a graph $\Gamma $ is a
sequence of edges $e_{1}$,\ldots ,$e_{n}$ such that $r(e_{i})=s(e_{i+1})$
for $i=1, 2, \ldots , n-1$. In this case we say that $s(p)=s(e_{1}),r(p)=r(e_{n}).$

If $W\subset V$ and $r(p)\in W$ then the path $p$ is said to be a path from the vertex $s(p)$ to the subset $W$.

\medskip

Let $Path(\Gamma)$ denote the set of all paths in the graph $\Gamma $.

\medskip

If $s(p)=r(p)$ then the path $p$ is closed. If $p=e_{1}\cdots e_{n}$ is
closed path and the vertices $s(e_{1}), \ldots ,   s(e_{n})$ are distinct
then we call the path $p$ a cycle. Denote $V(p)=\{s(e_{1})$, \ldots , $s(e_{n})\}$, $E(p)=\{e_{1}$,\ldots ,$e_{n}\}$.

An edge $e\in E$ is called an exit of a cycle $C$ if $s(e)\in V(C)$, but $e\notin E(C)$. A cycle that has no exits is called a $NE$-$cycle$.

\medskip

We say that a vertex $w$ is a descendant of a vertex $v$ if there exists a
path $p\in Path(\Gamma )$ such that $s(p)=v$, $r(p)=w$.

\medskip

A subset $W\subset V$ is called hereditary if for an arbitrary vertex $w\in W
$ all descendants of $w$ lie in $W$ (see [1, 2]). An empty subset of $V$ is
viewed as hereditary.

\medskip

Let $C$ be a cycle and let $W$ be a hereditary subset of $V$. If $V(C)\cap
W=\phi $ and all vertices in $V(C)$ have descendants in $W$ then we write $%
C\Rightarrow W$.

\medskip

Finally, for two nonempty subsets $X$, $Y\subset V$ let $E(X$, $Y)$

(respectively $Path(X$, $Y)$) denote the set of edges (respectively paths)
with the source in $X$ and the range in $Y$.

\bigskip

Let $\Gamma $ be a row-finite graph and let $F$ be a field. The Leavitt path
$F$-algebra $L(\Gamma )$ is the $F$-algebra presented by the set of
generators $\{v$ $|$ $v\in V\}$, $\{e$, $e^{\ast }|$ $e\in E\}$ and the set
of relations (1) $vw=\delta _{v,w}v$ for all vertices $v$, $w\in V$; (2) $%
s(e)e=er(e)=e$, $r(e)e^{\ast }=e^{\ast }s(e)=e^{\ast }$ for all $e\in E$;
(3) $e^{\ast }f=\delta_{e\text{,}f}r(e)$ for all $e$, $f\in E$; (4) $v=\underset{%
s(e)=v}{\sum }ee^{\ast }$, for an arbitrary vertex $v$, which is not a sink
(see [2, 6]).

\medskip

The mapping, which sends $v$ to $v$ for $v\in V$, $e$ to $e^{\ast }$ and $%
e^{\ast }$ to $e$ for $e\in E$, extends to an involution of the algebra $%
L(\Gamma )$. If $p=e_{1}\cdots e_{n}$ is a path, then $p^{\ast }=e_{n}^{\ast
}\cdots e_{1}^{\ast }$. In [7] A. Aranda Pino and K. Crow proved that the
center of a simple Leavitt path algebra is $(0)$ if the graph is infinite
and $F\cdot 1$ if the graph is finite.

\medskip

In [8] M. G. Corrales Garcia, D. M. Barquero, C Martin Gonzalez, M. Siles
Molina, J. F Solanilla Hernandez showed that for %an arbitrary infinite
%connected row-finite graph the center is $(0)$; for
a finite graph the center is isomorphic to a finite direct sum $F[t^{-1}$, $t]\oplus \cdots
\oplus F[t^{-1}$, $t]\oplus F\oplus \cdots \oplus F$ of the Laurent
polynomial algebra $F[t^{-1}$, $t]$ and the field $F$.

\medskip

In [9] the same authors found a highly nontrivial description of the center
in terms of extreme cycles of the graph.

\bigskip
\textbf{Important Remark:} Everywhere except in Section 4 we assume that the underlying graph $\Gamma $ is finite.

\medskip

Let $W$ be a nonempty subset of $V$. Consider the subset

\qquad \qquad \qquad $W^{\perp }=\{v\in V | Path(\{v\} , W)=\phi \}$.

\medskip

For the empty subset we let $\phi ^{\perp }=V$. It is easy to see that $%
W^{\perp }$ is a hereditary subset of $V$. If $W_{1}\subseteq W_{2}\subseteq
V$ then $W_{1}^{\perp }\supseteq W_{2}^{\perp }$, $(W^{\perp })^{\perp
}\supseteq W$, $W^{\perp }=((W^{\perp })^{\perp })^{\perp }$.

\medskip

We will refer to subsets $W^{\perp }$, $W\subseteq V$, as annihilator
hereditary subsets.

\medskip

Lemma 16 of [4] asserts that $(W^{\perp })^{\perp }$ is the largest
hereditary subset of $V$ such that from every vertex of $(W^{\perp })^{\perp
}$ there exists a path to $W$.

\medskip

Recall that a Boolean algebra is a set with two operations $\wedge $ and $%
\daleth $, which satisfy a certain list of axioms (see [11]).

\begin{example}\label{ex1}
 Let $X$ be a set. The set of  all subsets of $X$ is a Boolean algebra with respect to
the operations of intersection and complementation.
\end{example}

\begin{example}\label{ex2}
Let $A$ be a unital associative algebra. The set $E(A)$ of all central
idempotents is a Boolean algebra with respect to the operations $e_{1}\wedge
e_{2}=e_{1}e_{2}, \daleth e=1-e$.
\end{example}

\begin{lemma}\label{lem1}
The set $B(\Gamma )$ of all annihilator hereditary subsets of $V$ is a
Boolean algebra with respect to the operations $W_{1}\cap W_{2}$, $\daleth
W=W^{\perp }$.
\end{lemma}

\begin{proof}
If $W_{1}$, $W_{2}$ are annihilator hereditary subsets of $V$ then $ W_{1}=(W_{1}^{\perp })^{\perp }, W_{2}=(W_{2}^{\perp })^{\perp }$ and
$W_{1}\cap W_{2}=(W_{1}^{\perp }\cup W_{2}^{\perp })^{\perp }$. Hence $%
B(\Gamma )$ is closed with respect to intersection and the operation $%
\daleth W=W^{\perp }$. Verification of the axioms of a Boolean algebra is
straightforward (it will also follow from the Proposition 2 of \S 2).
\end{proof}

\bigskip

Let $W$ be a nonempty subset of $V$. A path $p=e_{1}\cdots e_{n}\in
Path(\Gamma )$, $e_{i}\in E$, is said to be an arrival path in $W$ if $%
r(p)\in W$, but none of the vertices $s(e_{1})$, \ldots , $s(e_{n})$ lies in
$W$. A vertex $w\in W$ is viewed as an arrival path in $W$ of length $0$.
Let $Arr(W)$ denote the set of all arrival paths in $W$. We let $Arr(\phi
)=\phi $.

\medskip

We call a hereditary set $W$ finitary if $|Arr(W)|$ $<\infty $. The empty
subset is viewed as finitary as well.

\bigskip

\begin{lemma}\label{lem2}
Finitary annihilator hereditary subsets form a Boolean subalgebra of $%
B(\Gamma )$.
\end{lemma}

\begin{proof}
Let the subsets $W_{1}$, $W_{2}\in B(\Gamma )$ be finitary. We need to show
that the subsets $W_{1}\cap W_{2}$, $W_{1}^{\perp }$ are finitary as well.
We claim that $Arr(W_{1}\cap W_{2})\subseteq Arr(W_{1})\cup Arr(W_{2})$.
Indeed, let $p\in Arr(W_{1}\cap W_{2})$. If $p=e_{1}\cdots e_{n}$, $e_{i}\in
E$, $p\notin Arr(W_{1})$ then one of the vertices $s(e_{1})$, \ldots , $%
s(e_{n})$ lies in $W_{1}$. Let $v_1=s(e_{i})\in W_{1}$, $1\leq i\leq n$.
Similarly, if $p\notin Arr(W_{2})$ then there exists a vertex $v_2=s(e_{j})\in
W_{2}$, $1\leq j\leq n$. If $j\geq i$ then $v_{2}\in W_{1}\cap W_{2}$. If $%
i\geq j$ then $v_{1}\in W_{1}\cap W_{2}$. In both cases we have got a
contradiction with $p$ being an arrival path in $W_{1}\cap W_{2}$.

\medskip

Now let $W$ be a finitary annihilator hereditary set. We will show that the
set $W^{\perp }$ is finitary. If this is not the case then there exists a
cycle $C$ such that $C\Rightarrow W^{\perp }$. Since the set $W$ is
hereditary and $W\cap W^{\perp }=\phi $ we conclude that $V(C)\cap W=\phi $.
Moreover, since the set $W$ is finitary it follows that $C\nRightarrow W$.
Hence all vertices from $V(C)$ lie in $W^{\perp }$, a contradiction. This
finishes the proof of the lemma.
\end{proof}

Let $W$ be a finitary hereditary subset of $V$. Consider the graph
$\Gamma (W)=(W$, $E(W$, $W))$. The subalgebra of $L(\Gamma )$ generated by $%
W $, $E(W$, $W)$, $E(W$, $W)^{\ast }$ is isomorphic to the Leavitt path algebra $L(\Gamma (W))$
(see [6, 10], see also the basis of $L(\Gamma )$ introduced in [5]). We
will denote this subalgebra of $L(\Gamma )$ as $L(W)$.

Consider the (diagonal) subalgebra $diagL(\Gamma )=\sum\limits_{v\in V} vL(\Gamma )v$. It is easy to see that the center of $L(\Gamma )$ lies in
$diagL(\Gamma )$. The mapping $\varphi _{W}:diagL(W)\rightarrow diagL(\Gamma) $, $a\mapsto\sum\limits_{p\in Arr(W)}pap^{\ast }$ is an embedding
of the subalgebra $diagL(W)$ in $diagL(\Gamma )$. We will call this
embedding diagonal. In the next section we will define a diagonal
embedding of $diagL(W)$ of a not necessarily finitary hereditary subset into
the topological completion $\widehat{L}(\Gamma )$ of $L(\Gamma )$. We will
also show that the diagonal embedding maps the center of $\widehat{L}(W)$ into the center of $\widehat{L}(\Gamma )$.

Denote $\varphi _{W}\left( \underset{w\in W}{\sum }w\right) =\underset{p\in
Arr(W)}{\sum }pp^{\ast }=e(W)$.

\bigskip

Now we are ready to formulate one of the main results of the paper.

\bigskip

\begin{theorem}\label{thm1}
Let $\Gamma$ be a finite graph. The mapping $W\rightarrow e(W)$ is the isomorphism from the Boolean
algebra of finitary annihilator hereditary subsets of $V$ to the Boolean
algebra of central idempotents of $L(\Gamma )$.
\end{theorem}

Let $W_{1}$, $W_{2}$, \ldots , $W_{k}$ be all distinct minimal hereditary
subsets of $V$. Consider a partial equivalence on the integer segment $[1$, $%
k]:$ $i\sim j$ if there exists a cycle $C$ in $\Gamma $ such that $%
C\Rightarrow W_{i}$ and $C\Rightarrow W_{j}$. Extend this partial
equivalence to an equivalence on $[1$, $k]$ by transitivity. In other words,
$i\sim j$ if there exists a sequence $i=i_{1}$, \ldots , $i_{r}=j$ and for
each $1\leq s\leq r-1$ there exists a cycle $C_{s}$ such that $%
C_{s}\Rightarrow W_{i_{s}}$, $C_{s}\Rightarrow W_{i_{s+1}}$.

Let $[1$, $k]=I_{1}\overset{\cdot }{\cup }\cdots \overset{\cdot }{\cup }%
I_{m} $ be the decomposition of $[1$, $k]$ into the union of disjoint
equivalence classes.

Denote $U_{i}=\left( \left( \underset{j\in I_{i}}{\cup }W_{j}\right) ^{\perp
}\right) ^{\perp }$.

\bigskip

\begin{lemma}\label{lem3}
\begin{itemize}
  \item [(a)] $U_{i}$ is a finitary hereditary subset;
  \item [(b)] $U_{i}\cap U_{j}=\phi $ for $1\leq i\neq j\leq m$.
\end{itemize}
\end{lemma}

\begin{proof}
\begin{itemize}
 \item [(a)] If the subset $U_{i}$ is not finitary then there exists a cycle $C$ such
that $C\Rightarrow U_{i}$.

\medskip

By [4, Lemma 16] for an arbitrary vertex $v\in U_{i}$ there exists a path $%
p\in Path(\Gamma )$ such that $s(p)=v$, $r(p)\in $ $\underset{j\in I_{i}}{%
\cup }W_{j}$. Hence $C\Rightarrow \underset{j\in I_{i}}{\cup }W_{j}$. This
implies that vertices from $V(C)$ do not lie in $\left( \underset{j\in I_{i}}%
{\cup }W_{j}\right) ^{\perp }$. Since vertices from $V(C)$ also do not lie
in $U_{i}=\left( \left( \underset{j\in I_{i}}{\cup }W_{j}\right) ^{\perp
}\right) ^{\perp }$ we conclude that $C\Rightarrow \left( \underset{j\in
I_{i}}{\cup }W_{j}\right) ^{\perp }$.

An arbitrary vertex in $V$ has a descendant in $W_{1}\cup \cdots \cup W_{k}$
(see [4 ]). Descendants of vertices from $\left( \underset{j\in I_{i}}{%
\cup }W_{j}\right) ^{\perp }$ lie in $\underset{j\in \lbrack 1\text{, }%
k]\backslash I_{i}}{\cup }W_{j}$.

Hence, $C\Rightarrow \underset{j\in I_{i}}{\cup }W_{j}$ and $C\Rightarrow $ $%
\underset{j\in \lbrack 1\text{, }k]\backslash I_{i}}{\cup }W_{j}$, which
contradicts our definition of the equivalence in $[1$, $k]$. This completes
the proof of the assertion (a).

\medskip

 \item [(b)] Now suppose that $v\in U_{i}\cap U_{j}$, $i\neq j$. By [4] the vertex $v$
has a descendant in $\underset{t\in I_{j}}{\cup }W_{t}$. But $\underset{t\in
I_{j}}{\cup }W_{t}\subseteq \left( \underset{q\in I_{i}}{\cup }W_{q}\right)
^{\perp }$, hence the vertex $v$ can not

lie in $\left( \left( \underset{q\in I_{i}}{\cup }W_{q}\right) ^{\perp
}\right) ^{\perp }=U_{i}$.
\end{itemize}
\end{proof}

\bigskip

\begin{theorem}\label{thm2}
Let $\Gamma$ be a finite graph.
\begin{itemize}
  \item [(a)] $Z(L(\Gamma ))=\overset{m}{\underset{i=1}{\oplus }}\varphi_{U_{i}}(Z(L(U_{i}))).$
  \item [(b)] If $U_{i}=((W_j)^{\perp})^\perp,$where $W_j=V(C),$  $C$ is a finitary NE-cycle, then $Z(L(U_{i}))\cong F[t^{-1}, t]$ (see [3]). Otherwise $Z(L(U_{i}))\cong F$.
\end{itemize}
\end{theorem}
\bigskip

In  Section 4 we describe the centers of Leavitt path algebras of row finite infinite graphs.

For a $NE$-cycle $C$ denote $L(C)=L(V(C)), \, Arr(C)=Arr(V(C)),$ where $V(C)$ is hereditary set of vertices lying on the cycle $C.$ We say that the cycle $C$ is finitary if $V(C)$ is finitary.

\begin{theorem}\label{thm2}
Let $\Gamma$ be a row finite graph. Then the center $Z(L(\Gamma))$ is the sum of subspaces of the following types:
\begin{itemize}
  \item [(a)] $Fe(W),$ where $W$ runs over finite hereditary finitary nonempty subsets of $V,$
  \item [(b)] $F$-span of $\{\sum\limits_{p\in Arr(C)}pzp^*\mid z\in Z(L(C))\},,$ where  $C$ runs over NE-cycle, which are finitary in $\Gamma.$
\end{itemize}
\end{theorem}

\begin{example}
Let $\Gamma=$ \includegraphics[width=.2\textwidth]{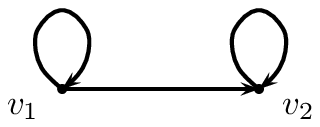} . The set $\{v_2\}$ is hereditary, but not finitary. Thus there are no proper finitary hereditary subsets and
$Z(L(\Gamma))=F\cdot1.$
\end{example}

\begin{example}
Let $\Gamma=$ \includegraphics[width=.1\textwidth]{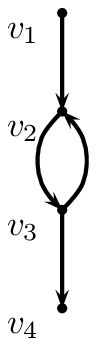} . The set $\{v_2,v_3,v_4\}$ is hereditary and  finitary, but $(\{v_2,v_3,v_4\}^\bot)^\bot=V.$ Thus there are no proper finitary hereditary annihilator subsets and again $Z(L(\Gamma))=F\cdot1.$
\end{example}

\begin{example}
Let $\Gamma=$ \includegraphics[width=.2\textwidth]{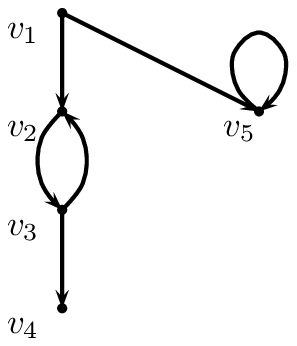} . The only finitary hereditary annihilator subsets are $\{v_5\}$ and  $\{v_2,v_3,v_4\}.$ Hence $Z(L(\Gamma))=F[t^{-1},t]\oplus F.$
\end{example}

\begin{example}
Let $\Gamma=$ \includegraphics[width=.3\textwidth]{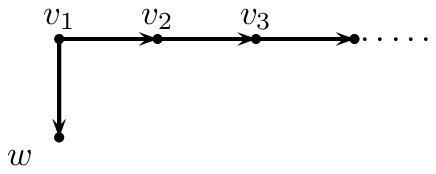} . The graph is an infinite graph and the only finite finitary hereditary subset is $\{w\}.$
Hence the center $Z(L(\Gamma))$ is one-dimensional and spanned by $e(\{w\})=w+ee^*.$
\end{example}

\section{ Centers of completed Leavitt path algebras}

\bigskip

We start with some definitions and results from [4].

\medskip

A mapping $\gamma :V\backslash \{sinks\}\rightarrow E$ is called a
specialization if $s(\gamma (v))=v$ for an arbitrary vertex $v\in
V\backslash \{sinks\}$. Edges lying in $\gamma (V\backslash \{sinks\})$ are
called special.

\medskip

For a specialization $\gamma $ consider the set $B(\gamma )$ of products $%
pq^{\ast }$, where $p=e_{1}\cdots e_{n}$, $q=f_{1}\cdots f_{m}$ are paths in
$\Gamma $; $e_{i}$, $f_{j}\in E$; $r(p)=r(q)$ and either $e_{n}\neq f_{m}$
or $e_{n}=f_{m}$, but this edge is not special.

\medskip

In [5] we proved that $B(\gamma )$ is a basis of the algebra $L(\Gamma )$.

\medskip

We call a \ path $p=e_{1}\cdots e_{n}$, $e_{i}\in E$, special if all edges $%
e_{1}$, \ldots , $e_{n}$ are special.

\medskip

For an arbitrary path $p=e_{1}\cdots e_{n}$ let $i$ be the minimal integer
such that the path $e_{i+1}\cdots e_{n}$ is special. If the edge $e_{n}$ is
not special then $i=n$. Let $sd(p)=n-i$.

\medskip

The algebra $L(\Gamma )$ is $\mathbb{Z}
$-graded: $\deg(V)=0, \deg(E)=1, \deg(E^*)=-1$.

\medskip

For nonnegative real numbers $n,\, s,\, d$ consider the subspace $V_{n,s,d}$
of $L(\Gamma )$ spanned by all products $pq^{\ast }$ such that $p$, $q\in
Path(\Gamma )$, $l(p)+l(q)\geq n$, $sd(p)+sd(q)\leq s$, $|deg(pq^{\ast })|=|l(p)-l(q)|\leq d$.

\medskip

For $k\geq 1$ let $V_{k}=\sum \{V_{n,s,d} | n\geq k(s+d+1)\}$.

\medskip

In [4] we proved that $\underset{k\geq 1}{\cap }V_{k}=(0)$. Hence, $%
\{V_{k}\}_{k\geq 1}$ can be viewed as a basis of neighborhoods of $0$. The
topology defined by $\{V_{k}\}_{k\geq 1}$ is compatible with algebraic
operations on $L(\Gamma )$, thus, $L(\Gamma )$ becomes a topological
algebra. Let $\overline{L(\Gamma )}$ be the completion of the topological
algebra $L(\Gamma )$. Let $\overline{L(\Gamma )_{i}}$ be the completion of
the homogeneous component $L(\Gamma )_{i}$ of degree $i$ in the algebra $%
\overline{L(\Gamma )}$. We call $\widehat{L}(\Gamma )=\underset{i\in
\mathbb{Z}
}{\sum }\overline{L(\Gamma )_{i}}$ the graded completion of $L(\Gamma )$.

In [4] it was shown that $B(\gamma )$ is a topological basis of the algebra $%
\widehat{L}(\Gamma )$. It implies that for a nonempty hereditary subset $%
W\subset V$ the identical mappings $W\rightarrow W$, $E(W$,$W)\rightarrow E(W
$,$W)$ extend to an embedding $\widehat{L}(W)\rightarrow \widehat{L}(\Gamma )
$.

\medskip

In [4] we proved that for an arbitrary nonempty hereditary subset $W\subset V
$ the sum $e(W)=\underset{p\in Arr(W)}{\sum }pp^{\ast }$ converges in $%
\widehat{L}(\Gamma )$ and $e(W)$ is a central idempotent.

Let $W_{1}$, \ldots , $W_{k}$ be all distinct minimal hereditary subsets of $%
V$. In [4] we proved existence of a specialization $\gamma :V\backslash
\{sinks\}\rightarrow E$ with the following properties:
\begin{itemize}
\item[1)] there are finitely many special paths $p=e_{1}\cdots e_{n}$, $e_{i}\in E,$ such that
$s(e_{1}), \ldots , s(e_{n})\in V\setminus \left( \overset{k}{\underset{i=1}{\cup }}W_{i}\right) $;
\item[2)] for each subset $W_{i}$, which does not consist of one sink the graph $%
(W_{i}$, $\gamma (W_{i}))$ is connected.
\end{itemize}

In [4] such specializations are called regular. From now on we assume that $%
\gamma :V\backslash \{sinks\}\rightarrow E$ is a regular specialization.

\bigskip

Now we are ready to define the diagonal embedding

$\varphi _{W}:$ $diag$ $\widehat{L}(W)\rightarrow diag$ $\widehat{L}(\Gamma )
$ for a not necessarily finitary hereditary subset $W\subset V$.

\bigskip

\begin{lemma}\label{lem4}
\begin{itemize}
  \item [(a)] For an arbitrary element $a\in \widehat{L}(W)$ the sum $\varphi_{W}(a)=\underset{p\in Arr(W)}{\sum }pap^{\ast }$ converges in $\widehat{L}(\Gamma )$;
  \item [(b)] the mapping $diag$ $\widehat{L}(W)\rightarrow diag$ $\widehat{L}(\Gamma) $, $a\mapsto \varphi_{W}(a)$, is an embedding;
  \item [(c)] $\varphi _{W}$ maps the center $Z(\widehat{L}(W))$ to the center $Z(\widehat{L}(\Gamma ))$.
\end{itemize}
\end{lemma}

\begin{proof}
\begin{itemize}
\item[(a)] Without a loss of generality we will assume that the element $a$ is
homogeneous, $deg(a)=d$. Let $a=\sum \alpha _{q_{1}q_{2}}q_{1}q_{2}^{\ast
}\in \widehat{L}(W)$; $\alpha _{q_{1}q_{2}}\in F;$ $q_{1}$, $q_{2}\in Path(W)
$ be a converging sum. Because of the property (1) of the specialization $%
\gamma $ there are finitely many special paths in $Arr(W)$. Let $s$ be the
maximum of lengths of these paths.

\medskip

Fix $k\geq 1$. We will show that all but finitely many elements $%
pq_{1}q_{2}^{\ast }p^{\ast }$, $p\in Arr(W)$, lie in $V_{k}$.

If $q_{1}q_{2}^{\ast }\in V_{k}$ and $l(p)\geq ks$ then $pq_{1}q_{2}^{\ast
}p^{\ast }\in V_{k}$.

Indeed, $l(q_{1})+l(q_{2})+2l(p)\geq k(sd(q_{1})+sd(q_{2})+2s+d+1)$. From
continuity of algebraic operations in $\widehat{L}(\Gamma )$ it follows that
there exists $k_{1}\geq k$ such that $pV_{k_1}p^{\ast }\subseteq V_{k}$ for
all arrival paths $p$ with $l(p)<ks$. Now $q_{1}q_{2}^{\ast }\in V_{k_{1}}$
implies $pq_{1}q_{2}^{\ast }p^{\ast }\in V_{k}$ for all $p\in Arr(W)$.

Let $\alpha _{q_{1}q_{2}}(i)q_{1}^{(i)}q_{2}^{(i)^{\ast }}$, $1\leq i\leq r$%
, be all summands in the decomposition of $a$ that do not lie in $V_{k_{1}}$.

\medskip

Let $l=\max\{\frac{1}{2}(k(sd(q_{1}^{(i)})+sd(q_{2}^{(i)})+2s+d+1)-l(q_{1}^{(i)})-l(q_{2}^{(i)})), 1\leq i\leq r\}$.

\medskip

If $p\in Arr(W)$ and $l(p)\geq l$ then $pq_{1}^{(i)}q_{2}^{(i)^{\ast
}}p^{\ast }\in V_{k}$ for $i=1, \ldots , r$. Hence, $p\in Arr(W)$, $%
l(p)\geq l$ implies that $pq_{1}q_{2}^{\ast }p^{\ast }\in V_{k}$ for all
summands in the decomposition of $a$.

Now the only elements $pq_{1}q_{2}^{\ast }p^{\ast }$, that may not lie in $%
V_{k}$ are those with $l(p)<l$, $q_{1}q_{2}^{\ast }\in
\{q_{1}^{(i)}q_{2}^{(i)^{\ast }}$, $1\leq i\leq r\}$. This proves the
assertion (a).

\medskip

\item[(b)] Let $p$, $q\in Path(\Gamma )$. We say that the path $p$ is a beginning
of the path $q$, while the path $q$ is a continuation of the path $p$ if
there exists a path $q^{\prime }\in Path(\Gamma )$ such that $q=pq^{\prime }$%
. We will use the following straightforward fact: if $p$, $q
\in Path(\Gamma )$ then $p^{\ast }q\neq 0$ if and only if one of the paths
$p$, $q$ is a continuation of the other.

\medskip

If $p$, $q\in Arr(W)$ and one of them is a continuation of the other one
then $p=q$.

\medskip

This implies that for arbitrary elements $a$, $b\in $ $diag\widehat{L}(W)$
we have

$\varphi _{W}(a)\varphi _{W}(b)=\sum pap^{\ast }qbq^{\ast }=\sum
par(p)bp^{\ast }=\sum pabp^{\ast }=\varphi _{W}(ab)$.

\medskip

Hence $\varphi _{W}:diag\widehat{L}(W)\rightarrow diag\widehat{L}(\Gamma )$
is a homomorphism. It is easy to see that this homomorphism is continuos.
Finally, for an arbitrary vertex $w\in W$ we have $w\varphi _{W}(a)w=waw$, $%
a=\underset{w\in W}{\sum }w\varphi _{W}(a)w$. It implies that the
homomorphism $\varphi _{W}:diag\widehat{L}(W)\rightarrow diag\widehat{L}%
(\Gamma )$ is an embedding.

\medskip

\item[(c)] Now let $a\in Z(\widehat{L}(W))$. We will show that $\varphi _{W}(a)\in
Z(\widehat{L}(\Gamma))$. Since the algebra $L(\Gamma )$ is dense in $\widehat{L}%
(\Gamma )$ it is sufficient to show that $\varphi _{W}(a)$ commutes with all
edges $e\in E$. Suppose at first that $r(e)\notin W$. Then $e\varphi _{W}(a)=%
\underset{p\in Arr(r(e),W)}{\sum }epap^{\ast }$, $\varphi _{W}(a)e=\sum
pap^{\ast }e$, where the sum is taken over all paths $p\in Arr(W)$ that are
continuations of the edge $e$. It is easy to see that $eArr(r(e),W)=\{p\in
Arr(W)$ $|$ $p$ is a continuation of $e\}$. This implies the result.

\medskip

Now suppose that $r(e)=w\in W$, but $s(e)\notin W$, thus $e\in Arr(W)$. Then
$e\varphi _{W}(a)=eaw$, $\varphi _{W}(a)e=eae^{\ast }e=eaw$.

\medskip

Finally, if $s(e)$, $r(e)\in W$ then $e\varphi _{W}(a)=ea=ae=\varphi
_{W}(a)e $ since $a\in Z(\widehat{L}(W))$.
\end{itemize}
 This finishes the proof of the
lemma.
\end{proof}

\bigskip

Let $I(W_{i})$ be the (closed) ideal generated by the set $W_{i}$ in $%
\widehat{L}(\Gamma )$. In [4] we proved that $\widehat{L}(\Gamma
)=I(W_{1})\oplus \cdots \oplus I(W_{k})$ and, for regular specializations $%
\gamma $, each summand $I(W_{i})$ is a topologically graded simple algebra.

\bigskip

\begin{lemma}\label{lem5}
 Let $W$ be a minimal hereditary subset of $V$. Then $Z(I(W))=\varphi
_{W}(Z(\widehat{L}(W))$.
\end{lemma}

\begin{proof}
Lemma 4(c) implies that $\varphi _{W}(Z(\widehat{L}(W))\subseteq Z(\widehat{L%
}(\Gamma ))$.

\medskip

An arbitrary element $z$ from $I(W)$ can be represented as a converging sum $%
z=\sum pa_{p,q}q^{\ast }$, where $p$, $q\in Arr(W)$; $a_{p,q}\in \widehat{L}%
(W)$. Suppose that $z\in Z(I(W))$. Then $z=\underset{v\in V}{\sum }vzv$. In
other words, we can assume that in each summand $s(p)=s(q)$.

Let $pa_{p,q}q^{\ast }$ be a nonzero summand such that $s(p)=s(q)\notin W$.
We have $p^{\ast }zq=a_{p,q}=p^{\ast }qz\neq 0$. Hence one of the paths $p$,
$q$ is a continuation of the other path. But both paths lie in $Arr(W)$.
This implies $p=q$. We proved that $z=\sum pa_{p}p^{\ast }$; $p\in Arr(W)$; $%
a_{p}\in \widehat{L}(W)$.

Let $pa_{p}p^{\ast }$ be a nonzero summand. We have $p^{\ast
}zp=a_{p}=p^{\ast }pz=r(p)z$. It is easy to see that the element $a=\underset%
{w\in W}{\sum }wz$ lies in the center of $\widehat{L}(W)$. We have $z=%
\underset{p\in Arr(W)}{\sum }pap^{\ast }=\varphi _{W}(a)$, which finishes
the proof of the lemma.
\end{proof}

\bigskip

Let $p\in Path(\Gamma).$ Let's decompose $pp^*$ as a linear combination of basic elements from $B(\gamma).$
If the last edge of the path $p$ is not special then $pp^*\in B(\gamma).$ Let $p=p'p'',$ the last edge of the path $p'$ is not special or $p'$ is a vertex,
$p''=e_1\cdots e_n$ is a special path. Denote $v_i=s(e_i),\, 1\leq i \leq n.$

For an arbitrary vertex $v\in V$ let $\mathcal{E}(v)$ denote the set of all non special edges with source at $v.$
Then $$pp^*=p'p'^*-\sum\limits_{k=1}^n p'e_1\cdots e_{k-1}\left(\sum\limits_{\scriptsize e\in\mathcal{E}(v_k)} ee^*\right)e^*_{k-1}\cdots e^*_1p'^*.$$

\begin{lemma}\label{lem6}
Let $p$ be a closed path of length $\geq 1$ such that $pp^{\ast }=p^{\ast
}p=r(p)$. Then $p=C^{m}$, where $C$ is an NE-cycle.
\end{lemma}

\begin{proof}
%Let $p=e_{1}\cdots e_{n}$, $e_{i}\in E$, $v_{i}=s(e_{i})$, $1\leq i\leq n$, $%
%r(e_{n})=s(e_{1})=v_{1}$. For a vertex $v\in V$ let $\varepsilon (v)$ denote
%the set of all non-special edges with the source at $v$. Then
%
%\qquad $pp^{\ast }=v_{1}$ - $\overset{n}{\underset{k=1}{\sum }}e_{1}\cdots
%e_{k-1}\left( \underset{e\in \varepsilon (v_{k})}{\sum }ee^{\ast }\right)
%e_{k-1}^{\ast }\cdots e_{1}^{\ast }$
%
%is the linear combination of basic elements from $B(\gamma )$. If $pp^{\ast
%}=p^{\ast }p=v_{1}$ then $\varepsilon (v_{i})=\phi $, $1\leq i\leq n$, hence
%for each vertex $v_{i}$ there is only one edge with the source at $v_{i}$.
%This implies the result.
Consider the above decomposition of the element $pp^*$ as a linear combination of basic elements from $B(\gamma).$
Since $pp^*=s(p)$ we conclude that $p'$ is a vertex, hence $p=e_1\cdots e_n$ is a special path. Moreover, the subsets $ \mathcal{E}(v_i),\, 1\leq i \leq n,$ are empty.
Hence for each vertex $v_i$ there is only one edge with the source at $v_i.$ This implies the result.
\end{proof}

\bigskip

\begin{lemma}\label{lem7}
Let $W$ be a minimal hereditary subset of $V$. If $Z(\widehat{L}(W))$ has
a nonzero homogeneous component of nonzero degree then $W=V(C)$, where $C$
is an NE-cycle.
\end{lemma}

\begin{proof}
Suppose that an element $0\neq z=\sum \alpha _{p,q}pq^{\ast }\in Z(\widehat{L%
}(W))$ has degree $d\geq 1$; $\alpha _{p,q}\in F$; $p$, $q\in Path(\Gamma )$%
, $pq^{\ast }\in B(\gamma )$, $s(p)=s(q)\in W$.

\medskip

Let $P$ denote the set of pairs

\qquad $P=\{(p$, $q)\in Path(\Gamma )\times Path(\Gamma )$ $|$ $\alpha
_{p,q}\neq 0\}$.

\medskip

For each pair $(p$, $q)\in P$ we have $l(p)-l(q)=d$. Choose a pair $(p_{0}$,
$q_{0})\in P$ such that the length $l(p_{0})$ is minimal.

\medskip

We claim that $l(q_{0})=0$. Indeed, suppose that for all pairs $(p$, $q)\in P
$ we have $l(q)\geq 1$. Then

\medskip

\qquad $p_{0}^{\ast }z=\underset{(p_{0},q_{i})\in P}{\sum }\alpha
_{p_{0,}q_{i}}q_{i}^{\ast }+\underset{l(p^{\prime })\geq 1}{\underset{%
(p_{0}p^{\prime },q)\in P}{\sum }}\alpha _{p_{0}p^{\prime },q}p^{\prime
}q^{\ast }$,

\qquad $zp_{0}^{\ast }=\sum \alpha _{p,q}p(p_{0}q)^{\ast }$.

\medskip

Elements $p^{\prime }q^{\ast }$ and nonzero elements $p(p_{0}q)^{\ast }$ lie
in the basis $B(\gamma )$. If $l(q)\geq 1$ for all $(p$, $q)\in P$ then the $%
q_{0}^{\ast }$ does not appear in the sums $\underset{l(p^{\prime })\geq 1}{%
\underset{(p_{0}p^{\prime },q)\in P}{\sum }}\alpha _{p_{0}p^{\prime
},q}p^{\prime }q^{\ast }$ and $\underset{(p\text{, }q)\in P}{\sum }\alpha
_{p,q}p(p_{0}q)^{\ast }$. Hence $l(q_{0})=0$. Now,

\medskip

$p_{0}^{\ast }z=\alpha v+\underset{l(p^{\prime })\geq 1\text{,}l(q)\geq 1}{%
\underset{(p_{0}p^{\prime },q)\in P}{\sum }}\alpha _{p_{0}p^{\prime
},q}p^{\prime }q^{\ast }$, $0\neq \alpha \in F$, $v=r(p_{0})$,

$zp_{0}^{\ast }=\alpha p_{0}p_{0}^{\ast }+\underset{p\neq p_{0}}{\sum }%
\alpha _{p,q}p(p_{0}q)^{\ast }$.

\medskip

This implies $p_{0}p_{0}^{\ast }=p_{0}^{\ast }p_{0}=v$. Now it remains to
refer to Lemma 6. to complete the proof of the lemma.
\end{proof}

\bigskip

If $W=V(C)$, where $C=e_{1}\cdots e_{n}$ is a NE-cycle, $|W|=n$, then an
algebra $\widehat{L}(W)=L(W)$ (see [4]) is isomorphic to the algebra of $%
n\times n$ matrices over Laurent polynomials ([3]). In this case the center
$$ Z(L(W))=\sum\limits_{i\geq 1} Fz^{i}+\sum\limits_{i\geq 1}F(z^{\ast })^{i}+\sum\limits_{v\in V(C)}v\cong F[t^{-1},t],$$
where $z=e_{1}\cdots e_{n}+e_{2}e_{3}\cdots e_{n}e_{1}+\cdots +e_{n}e_{1}\cdots e_{n-1}$.

\bigskip

\begin{lemma}\label{lem8}
 Let $W$ be a minimal hereditary subset of $V$, which is not a set of
vertices of an NE-cycle. Then $Z(\widehat{L}(W))=F\cdot \underset{w\in W}{%
\sum }w$.
\end{lemma}

\begin{proof}
By Lemma 7 all elements of $Z(\widehat{L}(W))$ have degree $0$. Let $0\neq
z\in Z(\widehat{L}(W))$. Choose a vertex $w\in W$. We have $zw=S_{1}+S_{2}$,
where $S_{1}$ is a finite linear combination of basic elements from $%
B(\gamma )$; $S_{2}=\underset{j}{\sum }\alpha _{j}p_{j}q_{j}^{\ast }$; $p_{j}
$, $q_{j}\in Path(\Gamma )$, $l(p_{j})=l(q_{j})$, $s(p_{j})=s(q_{j})=w$, $%
p_{j}q_{j}^{\ast }\in B(\gamma )$ and all summands $p_{j}q_{j}^{\ast }$ lie
in $V_{2}$.

\medskip

If the subset $W$ consists of one sink then the assertion is clear.

We will assume therefore that the subset $W$ does not contain sinks.

For a vertex $w\in W$ consider the sequence of special paths: $g_{w}(0)=w$, $%
g_{w}(n+1)=g_{w}(n)\cdot \gamma (r(g_{w}(n)))$. We claim that $%
g_{w}(n)g_{w}(n)^{\ast }S_{2}g_{w}(n)g_{w}(n)^{\ast }\rightarrow 0$ as $%
n\rightarrow \infty $.

Indeed, let $pq^{\ast }\in V_{2}$, and consider the element $%
g_{w}(n)g_{w}(n)^{\ast }pq^{\ast }g_{w}(n)g_{w}(n)^{\ast }$. The inclusion $%
pq^{\ast }\in V_{2}$ implies that $2l(p)\geq 2(s(p)+s(q)+1)$. Hence both
paths $p$, $q$ contain non-special edges. Hence the path $g_{w}(n)$ can not
be a continuation of the path $p$ (resp. $q$). If $p$ and $q$ are extensions
of the path $g_{w}(n)$ then $g_{w}(n)g_{w}(n)^{\ast }pq^{\ast
}g_{w}(n)g_{w}(n)^{\ast }=pq^{\ast }$.

\medskip

It implies that $g_{w}(n)g_{w}(n)^{\ast }S_{2}g_{w}(n)g_{w}(n)^{\ast }=%
\underset{j}{\sum }\alpha _{j}p_{j}q_{j}^{\ast }$, the summation is taken
over those elements, where both $p_{j}$, $q_{j}$ are extensions of $g_{w}(n)$%
. Clearly, this sequence converges to $0$ as $n\rightarrow \infty $.

Now consider the element $S_{1}=\overset{r}{\underset{i=1}{\sum }}\alpha
_{i}p_{i}q_{i}^{\ast }$, $l(p_{i})=l(q_{i})$, $\alpha _{i}\in F$.

Let $n\geq $ $max\ (l(p_{i})$, $1\leq i\leq r)$. For each summand we have $%
g_{w}(n)^{\ast }p_{i}q_{i}^{\ast }g_{w}(n)=0$ unless $g_{w}(n)$ is an
extension of the path $p_{i}=q_{i}$, in which case $g_{w}(n)^{\ast
}p_{i}q_{i}^{\ast }g_{w}(n)=r(g_{w}(n))$. Let $\alpha =\sum \{\alpha _{j}|$ $%
g_{w}(n)$ is an extension of $p_{i}=q_{i}\}$. Then $g_{w}(n)^{\ast
}S_{1}g_{w}(n)=\alpha r(g_{w}(n))$ and $g_{w}(n)g_{w}(n)^{\ast
}S_{1}g_{w}(n)g_{w}(n)^{\ast }=\alpha g_{w}(n)g_{w}(n)^{\ast }$.

\medskip

In [4] we proved that $e_{w}=\underset{n\rightarrow \infty }{lim}%
g_{w}(n)g_{w}(n)^{\ast }$ is a nonzero idempotent in $\widehat{L}(W)$. From
what we proved above it follows that $\underset{n\rightarrow \infty }{lim}%
g_{w}(n)g_{w}(n)^{\ast }zg_{w}(n)g_{w}(n)^{\ast }=\alpha $ $lim$ $%
g_{w}(n)g_{w}(n)^{\ast }$, $e_{w}z=\alpha e_{w}$, $e_{w}(z-\alpha \cdot 1)=0$%
. Since $z-\alpha \cdot 1\in Z(\widehat{L}(W))$ and $\widehat{L}(W)$ is a
topologically graded simple algebra we conclude that $z=\alpha \cdot 1$,
which completes the proof of the lemma.
\end{proof}

\bigskip

The decomposition $\widehat{L}(W)=I(W_{1})\oplus \cdots \oplus I(W_{k})$
together with Lemmas 7, 8 implies the following description of the center of
$\widehat{L}(\Gamma )$.

\bigskip

\begin{prop}
\begin{itemize}
 \item[(a)] $Z(\widehat{L}(\Gamma ))=\overset{k}{\underset{i=1}{\oplus }}\varphi_{W_{i}}(Z(\widehat{L}(W_{i})))$;
\item[(b)] the center $Z(\widehat{L}(W_{i}))$ is isomorphic to the algebra of
Laurent polynomials if $W_{i}$ is an NE-cycle. Otherwise $Z(\widehat{L}(W_{i}))\cong F$.
\end{itemize}
\end{prop}
\bigskip

Now we will discuss connections between the Boolean algebra of annihilator
hereditary subsets of $V$ and central idempotents of $\widehat{L}(\Gamma )$.

\bigskip

\begin{lemma}\label{lem9}
The Boolean algebra of annihilator hereditary subsets of $V$ consists of $%
2^{k}$ subsets $((W_{i_{1}}\overset{\cdot }{\cup }\cdots \overset{\cdot }{%
\cup }W_{i_{s}})^{\perp })^{\perp }$, $1\leq i_{1}<\ldots <i_{s}\leq k$.
\end{lemma}

\begin{proof}
Let $W$ be a nonempty annihilator hereditary subset of $V$. Let $W_{i_{1}}$,
\ldots , $W_{i_{s}}$, $1\leq i_{1}<\ldots <i_{s}\leq k$, be all minimal
hereditary subsets contained in $W$. We claim that $W=((W_{i_{1}}\cup \cdots
\cup W_{i_{s}})^{\perp })^{\perp }$. Indeed, since an arbitrary vertex from $%
W$ has a descendant in $W_{1}\cup \cdots \cup W_{k}$ it follows that an
arbitrary vertex from $W$ has a descendant in $W_{i_{1}}\cup \cdots \cup
W_{i_{s}}$. By Lemma 16 of [4] it implies $W\subseteq ((W_{i_{1}}\cup \cdots
\cup W_{i_{s}})^{\perp })^{\perp }$.

\medskip

If $v\in V$ and there does not exist a path $p\in Path(\Gamma )$ such that $%
s(p)=v$, $r(p)\in W$ then there does not exist a path $p^{\prime }\in
Path(\Gamma )$ such that $s(p^{\prime })=v$, $r(p^{\prime })\in
W_{i_{1}}\cup \cdots \cup W_{i_{s}}$. Hence $W^{\perp }\subseteq
(W_{i_{1}}\cup \cdots \cup W_{i_{s}})^{\perp }$ and therefore ($W^{\perp
})^{\perp }\supseteq ((W_{i_{1}}\cup \cdots \cup W_{i_{s}})^{\perp })^{\perp
}$. Since $W$ is an annihilator subset we have $W=(W^{\perp })^{\perp }$, $%
W\supseteq ((W_{i_{1}}\cup \cdots \cup W_{i_{s}})^{\perp })^{\perp }$. This
finishes the proof of the lemma.
\end{proof}

\bigskip

\begin{lemma}\label{lem10}
For arbitrary hereditary subsets $W_{1}$, $W_{2}$ of $V$ we have $e(W_{1})e(W_{2})=e(W_{1}\cap W_{2})$.
\end{lemma}

\begin{proof}
We have $e(W_{1})e(W_{2})=S_{1}+S_{2}+S_{3}$,

$S_{1}=\sum \{p_{1}p_{1}^{\ast }p_{2}p_{2}^{\ast } | p_{i}\in Arr(W_{i}), p_{2} \text{ is a proper extension of } p_{1}\}$,

\medskip

$S_{2}=\sum \{p_{1}p_{1}^{\ast }p_{2}p_{2}^{\ast }|p_{i}\in Arr(W_{i}),  p_{1} \text{ is a proper extension of } p_{2}\}$,

\medskip

$S_{3}=\sum \{p_{1}p_{1}^{\ast }p_{2}p_{2}^{\ast } | p_{i}\in Arr(W_{i}), p_{1}=p_{2}\}$.

\medskip

It is easy to see that

\medskip

$S_{1}=\sum \{pp^{\ast } | p\in Arr(W_{1}\cap W_{2}), \text{ the path $p$ enters $W_{1}$ before it enters } W_{1}\cap W_{2}\}$;

\medskip

$S_{2}=\sum \{pp^{\ast } | p\in Arr(W_{1}\cap W_{2}), \text{ the path $p$ enters $W_{2}$ before it enters } W_{1}\cap W_{2}\}$;

\medskip

$S_{3}=\sum \{pp^{\ast } | p\in Arr(W_{1}\cap W_{2}), r(p) \text{ is the first vertex on $p$ lying in } W_{1}\cup W_{2}\}$.

\medskip

Since every path from $Arr(W_{1}\cap W_{2})$ falls in one of the categories
above, the assertion of the lemma follows.
\end{proof}

\bigskip

Recall that $B(\Gamma )$ denotes the Boolean algebra of annihilator
hereditary subsets of $V$.

\bigskip

\begin{prop}\label{pro2}
The mapping $W\rightarrow e(W)$, $W\in B(\Gamma )$, is an isomorphism
from the Boolean algebra $B(\Gamma )$ to the Boolean algebra of central
idempotents of $\widehat{L}(\Gamma )$.
\end{prop}

\begin{proof}
In [4] it was shown that the elements $e(W)$ are central idempotents of $%
\widehat{L}(\Gamma )$ (see also Lemma 4(c)). If $W=((W_{i_{1}}\cup \cdots
\cup W_{i_{s}})^{\perp })^{\perp }$, $1\leq i_{1}<\ldots <i_{s}\leq k$, then
by ([4], Lemma 16) we have $e(W)=e(W_{i_{1}}\cup \cdots \cup
W_{i_{s}})=e(W_{i_{1}})+\cdots +e(W_{i_{s}})$.

\medskip

The idempotent $e(W_{i})$ is the identity of the algebra $I(W_{i})$. Since
the algebra $I(W_{i})$ is topologically graded simple it follows that $%
I(W_{i})$ does not contain proper central idempotents. Hence all central
idempotents of the algebra $\widehat{L}(\Gamma )$ are sums $%
e(W_{i_{1}})+\cdots +e(W_{i_{s}})$, $1\leq i_{1}<\ldots <i_{s}\leq k$.

\medskip

We established a bijection between $Bl(\Gamma )$ and the set of central
idempotents of $\widehat{L}(\Gamma )$. To prove that it is a Boolean
isomorphism we refer to Lemma 10 and the fact that the annihilator of $%
((W_{i_{1}}\cup \cdots \cup W_{i_{s}})^{\perp })^{\perp }$ is $%
((W_{j_{1}}\cup \cdots \cup W_{j_{r}})^{\perp })^{\perp }$, where $\{j_{1}$,
\ldots , $j_{r}\}=\{1$, \ldots , $k\}\backslash \{i_{1}$, \ldots , $i_{s}\}$%
. This finishes the proof of Proposition 2.
\end{proof}

\bigskip

\section{ Center of the Leavitt path algebra $L(\Gamma )$.}

\bigskip

By Proposition 1 we have $Z(L(\Gamma ))\subseteq \oplus \varphi
_{W_{i}}(Z(\widehat{L}(W_{i}))$. Consider an element $a=\underset{i=1}{\overset{k}{%
\sum }}\varphi _{W_{i}}(a_{i})\in L(\Gamma )$, $a_{i}\in Z(\widehat{L}(W_{i})$.

\bigskip

\begin{lemma}\label{lem11}
Suppose that $a$ is a homogeneous element of degree $d\neq 0$ and $%
a_{i}\neq 0$. Then $W_{i}=V(C_{i})$, where $C_{i}$ is an NE-cycle and $%
|Arr(W_{i})|<\infty $.
\end{lemma}

\begin{proof}
Suppose that $d\geq 1$. From Lemma 8 it follows that $W_{i}=V(C_{i})$, $%
C_{i}=e_{1}\cdots e_{n}$, $e_{i}\in E$, is an NE-cycle, $a_{i}=\alpha (e_{1}\cdots
e_{n})^{r}+(e_{2}\cdots e_{n}e_{1})^{r}+\cdots +\cdots (e_{n}e_{1}\cdots
e_{n-1})^{r}$, $0\neq \alpha\in F,$ $d=nr$.

For an arbitrary path $p\in Arr(W_{i})$ the elements $p(e_{j}\cdots
e_{n}e_{1}\cdots e_{j-1})^{r}p^{\ast }$ lie in $B(\gamma )$. Moreover, these
elements won't cancel with any of the basic elements from the decompositions
of $\varphi _{W_{j}}(a_{j})$, $j\neq i$. If the set $Arr(W_{i})$ is infinite
then $\underset{p\in Arr(W_{i})}{\sum }pa_{i}p^{\ast }$ is an infinite sum,
which implies that $a\notin L(\Gamma )$. This completes the proof.
\end{proof}

\bigskip

Without loss of generality we will assume that $W_{1}$, \ldots , $W_{s}$, $%
s\leq k$, are all finitary NE-cycles among minimal hereditary subsets.

\bigskip

\begin{corollary}\label{cor1}
$$Z(L(\Gamma ))=\left( \oplus_{i=1}^{s}\varphi_{W_{i}}(Z(W_{i})\right) \oplus \left(Z(L(\Gamma ))\cap \sum\limits_{i=s+1}^{k} Fe(W_{i})\right).$$
\end{corollary}

Recall that for $1\leq i$, $j\leq k$ we let $i\sim j$ if there exists
a cycle $C$ in $\Gamma $ such that $C\Rightarrow W_{i}$, $C\Rightarrow W_{j}.$%
Then $\sim $ is extended to an equivalence on $[1$, $k]$ via transitivity.

If $W_{i}=V(C)$, where $C$ is an NE-cycle and $|Arr(W_{i})|$ $<\infty $ then
$\{i\}$ is an  equivalence class in its own.

\bigskip

\begin{lemma}\label{lem12}
 Let $C$ be a cycle, $a=\overset{m}{\underset{i=1}{\sum }}\alpha
_{i}p_{i}p_{i}^{\ast }$, $\alpha _{i}\in F$, $p_{i}\in Path(\Gamma )$. If $%
n\geq $ $\underset{1\leq i\leq m}{\text{max}}l(p_{i})$, then $C^{\ast
^{n}}aC^{n}=\alpha r(C)$, $\alpha \in F$.
\end{lemma}

\begin{proof}
Let $p\in Path(\Gamma )$, $n\geq l(p)$. Then $C^{\ast ^{n}}pp^{\ast
}C^{n}\neq 0$ if and only if $C^{n}$ is a continuation of the path $p$, $%
C^{n}=pp_{1}$, $p_{1}\in Path(\Gamma )$. Then $C^{\ast ^{n}}pp^{\ast
}C^{n}=p_{1}^{\ast }p^{\ast }pp^{\ast }pp_{1}=p_{1}^{\ast
}p_{1}=r(p_{1})=r(C)$, which completes the proof.
\end{proof}

\bigskip

We call a path acyclic if it does not contain any cycles.

Let $W$ be an arbitrary nonempty hereditary subset of $V$. Let $Arr(W)_{0}$
be the set of all acyclic paths from $Arr(W)$. All vertices from $W$ lie in $%
Arr(W)_{0}$.

\medskip

Let $p$ be an acyclic path such that $r(p)\notin W$ and let $C$ be a cycle,
such that $s(C)=r(C)=r(p)$. Consider the elements

\medskip

$a_{p,C}(W)=\sum \{qq^{\ast }|$ $q\in Arr(r(p)$, $W)$, $q$ is not a
continuation of $C\}$, and $b_{p,C}(W)=\underset{i\geq 1}{\sum }%
pC^{i}a_{p,C}C^{\ast ^{i}}p^{\ast }$. Remark that both sums are subsums of $%
\underset{q\in Arr(W)}{\sum }qq^{\ast }$. Hence they converge in $\widehat{L}%
(\Gamma )$ (see [4]).

\bigskip

Let $p=e_{1}\cdots e_{n}$ be an arrival path in $W$ of length $n\geq 1$ that
is not acyclic. It means that the vertices $v_{1}=s(e_{1})$, \ldots , $%
v_{n}=s(e_{n})$, $v_{n+1}=r(e_{n})$ are not all distinct.

Let $v_{i}=v_{j}$, $1\leq i<j\leq n+1$ and $j$ is minimal with this
property. Then the path $p_{1}=e_{1}\cdots e_{i-1}$ is acyclic, $%
r(e_{i-1})\notin W$. If $i=1$ then we let $p_{1}=v_{1}$. The path $%
C=e_{i}\cdots e_{j-1}$ is a first  cycle that occurs on the path $p.$. Let $k\geq 1$ be a maximal integer such
that $p$ is a continuation of $p_{1}C^{k}$. Then $p=p_{1}C^{k}p_{2}$, $%
p_{2}\in Arr(r(p)$, $W)$, $p_{2}$ is not a continuation of $C$. We showed
that

\medskip

$e(W)=\sum\limits_{\scriptsize q\in Arr(W)_{0}}qq^* +\sum\limits_{(p,C)} b_{p,C}(W)$.

\bigskip

\begin{lemma}\label{lem13}
\begin{itemize}
 \item[(a)] If $(p, C)\neq (p', C' )$, where $p, p'$ are acyclic paths; $C, C' $ are cycles,
$r(p), r(p')\notin W$, then $C^*p^*b_{p',C'}(W)pC=0$;
\item[(b)] if $q\in Arr(W)_{0}$ then $C^*p^*qq^*pC=0$;
\item[(c)] if $e(W)\in L(\Gamma )$ then for a sufficiently large $n\geq 1$ we have ${C^*}^np^*e(W)pC^{n}=r(C)$.
\end{itemize}
\end{lemma}

\begin{proof}
\begin{itemize}
\item[(a)] %We have $C^{\ast }p^{\ast }b_{p^{\prime },C^{\prime }}(W)pC=\underset{%
%j\geq 1}{\sum }C^{\ast }p^{\ast }p^{\prime }C^{\prime ^{j}}a_{p^{\prime
%},c^{\prime }}C^{\prime ^{\ast ^{j}}}p^{\prime ^{\ast }}pC$. Notice that $%
%C^{\ast }p^{\ast }p^{\prime }C^{\prime }=0$ as none of the paths $pC$, $%
%p^{\prime }C^{\prime }$ is a continuation of the other path;
Let us show that none of the paths $pC, p'C'$ is a continuation of the other path.
Suppose that $p'C'$ is a continuation of $pC.$ If $l(p'C')>l(pC)$ then the cycle $C$ in $p'C'$ occurs before the cycle $C',$ which contradicts the property of $C'.$ Hence $p'C'=pC.$
Let $l(p')>l(p),$ $p'=pp'',$ $l(p'')\geq 1.$ Then $r(p)=r(C)=r(C')=r(p').$ Hence $p''$ is a closed path, which contains a cycle. This contradicts acyclicity of the path $p'.$ we showed that $(p,C)=(p',C'),$ which contradicts our assumption.

In view of the above, $C^*p^*p'C'=0$ and therefore $C^*p^*b_{p',C'}(W)p'C'=\sum\limits_{j\geq 1} C^*p^*p'(C')^{j} a_{p',C'} (C'^{*})^j p'^*pC=0.$

\item[(b)] We have $C^{\ast }p^{\ast }q=0$ as none of the paths $pC$, $q$ is a
continuation of the other path. Indeed, $q$ is not a continuation of $pC$
since the path $q$ is acyclic. The path $pC$ is not a continuation of the
path $q$ since $r(q)\in W$ whereas $r(pC)\notin W$;

\item[(c)] Suppose that $e(W)\in L(\Gamma )$. Then by Lemma 12 for all sufficiently
large $n$ we have $C^{\ast ^{n}}p^{\ast }e(W)pC^{n}=\alpha r(C)$. We claim
that $\alpha =1$. Indeed, since the idempotent $e(W)$ lies in the center of $%
L(\Gamma )$ it follows that $C^{\ast ^{n}}p^{\ast }e(W)pC^{n}=e(W)r(C)=%
\underset{q\in Arr(r(C)\text{, }W)}{\sum }qq^{\ast }\neq 0$. Since $\alpha
^{2}r(c)=(e(W)r(c))^{2}=\alpha r(c),$  $\alpha^2=\alpha,$ we conclude that $\alpha =1$.
\end{itemize}
\end{proof}

\bigskip

\begin{lemma}\label{lem14}
 Suppose that $e(W)\in L(\Gamma )$ and $b_{p\text{,}C}\neq 0$. Then $b_{p%
\text{,}C}(W)=pCC^{\ast }p^{\ast }$.
\end{lemma}

\begin{proof}
As we have seen above $e(W)=\underset{q\in Arr(W)_{0}}{\sum }qq^{\ast }+$ $%
\underset{(p^{\prime },C^{\prime })}{\sum }b_{p^{\prime }\text{,}C^{\prime
}}(W)$, where the summation runs over all pairs $(p^{\prime },$ $C^{\prime })
$ such that $b_{p^{\prime }\text{,}C^{\prime }}(W)\neq 0$. By Lemma 13 for a
sufficiently large $n\geq 1$ we have $C^{\ast ^{n}}p^{\ast }e(W)pC^{n}=r(C)$%
, $C^{\ast }p^{\ast }\underset{q\in Arr(W)_{0}}{\sum }qq^{\ast }pC=0$, $%
C^{\ast }p^{\ast }$ $\underset{(p^{\prime },C^{\prime })}{\sum }b_{p^{\prime
}\text{,}C^{\prime }}(W)pC=C^{\ast }p^{\ast }b_{p,C}pC$.

Hence, $r(C)=C^{\ast ^{n}}p^{\ast }b_{p,C}(W)pC^{n}=\underset{i\geq 0}{\sum }%
C^{i}a_{p,C}(W)C^{\ast ^{i}}$. Applying the linear transformation $%
x\rightarrow x-cxc^{\ast }$ to both sides we get $r(C)-CC^{\ast }=a_{p,C}(W)$%
.

Now, $b_{p,C}(W)=p\left( \sum\limits_{i=1}^{\infty}
C^{i}(r(C)-CC^*)C^{*^{i}}\right)p^* =pCC^{\ast }p^{\ast }$, which
completes the proof of the lemma.
\end{proof}

\bigskip

\begin{lemma}\label{lem15}
Let $W^{\prime }$, $W^{\prime \prime }$ be nonempty hereditary subsets
of $V$, $W^{\prime }\cap W^{\prime \prime }=\phi $. Suppose that there
exists a cycle $C$ such that $C\Rightarrow W^{\prime }$ and $C\Rightarrow
W^{\prime \prime }$. Then the idempotents $e(W^{\prime })$, $e(W^{\prime
\prime })$ can not both lie in $L(\Gamma )$.
\end{lemma}

\begin{proof}
Let $v=s(C)=r(C)$. By the assumptions $Arr(v$, $W^{\prime })\neq \phi $. Let $p=v$, a path of length $0$. Then $b_{p, C}(W')\neq 0$. By
Lemma 14 $b_{p, C}(W^{\prime })=CC^{\ast }$. Similarly, $b_{p\text{, }%
C}(W^{\prime \prime })\neq 0$ and therefore $b_{p\text{, }C}(W^{\prime
\prime })=CC^{\ast }$. Finally, $b_{p\text{, }C}(W^{\prime }\cup W^{\prime
\prime })\neq 0$ and therefore $b_{p\text{, }C}(W^{\prime }\cup W^{\prime
\prime })=CC^{\ast }$. Now, $b_{p\text{, }C}(W^{\prime }\cup W^{\prime
\prime })=b_{p\text{, }C}(W^{\prime })+b_{p\text{, }C}(W^{\prime \prime })$,
$CC^{\ast }=CC^{\ast }+CC^{\ast }$, a contradiction.
\end{proof}

\bigskip

\begin{lemma}\label{lem16}
Let $\underset{i=1}{\overset{k}{\sum }}\alpha _{i}e(W_{i})\in L(\Gamma )$%
, $\alpha _{i}=0$ or $1$. If $\alpha _{i}=1$ and $i\sim j$ then $\alpha
_{j}=1$.
\end{lemma}

\begin{proof}
Consider the hereditary sets $W^{\prime }=\cup \{W_{i}|$ $\alpha _{i}=1\}$, $%
W^{\prime \prime }=\cup \{W_{i}|$ $\alpha _{i}=0\}$.

Suppose that  there exists a cycle $C$ such that $C\Rightarrow W_{i}$, $%
C\Rightarrow W_{j}$. Then $C\Rightarrow W^{\prime }$, $C\Rightarrow
W^{\prime \prime }$. By our assumption $e(W^{\prime })=\underset{i=1}{%
\overset{k}{\sum }}\alpha _{i}e(W_{i})$ lies in $L(\Gamma )$. Hence $%
e(W^{\prime \prime })=1-e(W^{\prime })$ also lies in $L(\Gamma )$, which
contradicts Lemma 15.
\end{proof}

\bigskip

Recall that $U_{i}=((\underset{j\in I_{i}}{\cup }W_{j})^{\perp })^{\perp }$,
$1\leq i\leq m$, are nonintersecting finitary annihilator subsets of $V$
(see \S 1, Lemma 3).

\bigskip

\begin{lemma}\label{lem17}
 An arbitrary central idempotent in $L(\Gamma )$ is of the type $%
e(U_{i_{1}})+\cdots +e(U_{i_{r}})$, $1\leq i_{1}<\ldots <i_{r}\leq m$.
\end{lemma}

\begin{proof}
A central idempotent $e\in L(\Gamma )$ lies in the center of $\widehat{L}%
(\Gamma )$. By Proposition 1 we have $e=\underset{i=1}{\overset{k}{\sum }}%
\alpha _{i}e(W_{i})$, $\alpha _{i}=0$ or $1$. By Lemma 16 $e=e\left(
\underset{j\in I_{i_{1}}}{\cup }W_{j}\right) +\cdots +e\left( \underset{j\in
I_{i_{r}}}{\cup }W_{j}\right) $, $1\leq i_{1}<\ldots <i_{r}\leq m$. Now it
remains to recall that $e\left( \underset{j\in I_{i}}{\cup }W_{j}\right)
=e(U_{i})$.
\end{proof}

\bigskip

\begin{lemma}\label{lem18}
 An arbitrary finitary annihilator hereditary subset of $V$ is of the
type $U_{i_{1}}\cup \cdots U_{i_{r}}$, $1\leq i_{1}<\ldots <i_{r}\leq m$.
\end{lemma}

\begin{proof}
By Lemma 9 an arbitrary annihilator hereditary subset $W$ of $V$ is of the
type

$W=((W_{i_{1}}\cup \cdots \cup W_{i_{s}})^{\perp })^{\perp }$, $1\leq
i_{1}<\ldots <i_{s}\leq k$.

By Lemma 16 the set $\{i_{1}$, \ldots , $i_{s}\}$ is a union of equivalence
classes. This implies the assertion of the lemma.
\end{proof}

\bigskip

\begin{proof}[Proof of Theorem 1.]

Lemmas 17, 18 establish a $1-1$ correspondence between finitary annihilator
hereditary subsets of $V$ and central idempotents of $L(\Gamma )$. Lemma 10
shows that this correspondence preserves the operations.
\end{proof}

\bigskip

\begin{proof}[Proof of Theorem 2.]

\medskip

Let $W_{1}$, \ldots , $W_{s}$, $s\leq k$, be all finitary NE-cycles of $%
\Gamma $. Then $U_{i}=(W_{i}^{\perp })^{\perp }$ is the saturation of the
set $W_{i}$, that is, the minimal hereditary saturated subset of $V$
containing $W_{i}$, and $\varphi _{W_{i}}(Z(W_{i}))=\varphi
_{U_{i}}(Z(U_{i}))\cong F[t^{-1}$, $t]$.

\medskip

By the Corollary of Lemma 11 and Lemma 17 the center $Z(L(\Gamma ))$ is a
direct sum of $\varphi _{U_{i}}(Z(L(U_{i}))$, $1\leq i\leq s$, and of $1$%
-dimensional ideals $Fe(U_{i})$, where $U_{i}$ correspond to equivalence
classes in $[s+1$, $k]$. This completes the proof of the theorem.
\end{proof}
\bigskip

\section{ Infinite Graphs}
In this section we no longer assume that the row finite graph $\Gamma$ is finite.

Let $z$ be a nonzero element from the center $Z(L(\Gamma)).$
Consider the subset $W=\{v\in V\mid zv=0\}.$
It is easy to see that the subset $W$ is hereditary and saturated. Clearly, $V\neq W$ since otherwise $z^2=0$ contradicting semiprimeness of the algebra $L(\Gamma)$ ( see [1] ).
For distinct vertices $v_1,v_2\in V$ we have $v_1zv_2=zv_1v_2=0.$ Hence $z=\sum\limits_{v\in V\setminus W} zv.$ This implies that $|V\setminus W|<\infty.$

As above, for a cycle $C=e_1\ldots e_n,\, e\in E,\, s(e_1)=r(e_n),$ we denote $z(C)=e_1e_2\ldots e_n+e_2e_3\ldots e_2+\cdots +e_n e_1\ldots e_{n-1}.$
Since the center is a graded subspace of $L(\Gamma)$ we assume the element $z$ to be homogeneous.

If $\deg(z)=d\geq 1$ then by Theorem 2 there exist disjoint finitary cycles $C_1,\ldots, C_r$ in $V\setminus W$ of length $n_1,\ldots, n_r$ respectively, such that
$$z=\sum\limits_{i=1}^r \alpha_i \left(\sum\limits_{p\in Arr(C_i)} pz(C_i)^{d/n_i}p^*\right) mod\, I(W), \, 0\neq  \alpha_i \in F.$$

If $\deg(z)=-d,\, d\geq 1, $ then $$z=\sum\limits_{i=1}^r \alpha_i \left(\sum\limits_{p\in Arr(C_i)} p(z(C_i)^*)^{d/n_i}p^*\right) mod\, I(W), \, 0\neq  \alpha_i \in F.$$

If $\deg(z)=0$ then there exist disjoint finitary hereditary (in $\Gamma/ W$) subsets $U_1,\ldots, U_r\subseteq V\setminus W$ such that

$$z=\sum\limits_{i=1}^r \alpha_i \left(\sum\limits_{p\in Arr(U_i)} pp^*\right) mod\, I(W), \, 0\neq  \alpha_i \in F.$$

We say that a vertex is cyclic if it lies on a closed path of the graph.

\begin{lemma}\label{lem19}
 Let $v\in V\setminus W$ be cyclic vertex that lies in $\cup_{i=1}^r V(C_i)$ (if $\deg(z)\neq0$) or in $U_1\dot{\cup }\cdots \dot{\cup} U_r$ ($\deg(z)=0).$
 Then $v\in W^\perp.$
\end{lemma}

\begin{proof}
 We will discuss only the case $\deg(z)=d\geq 1.$
 The cases $\deg(z)<0$ and $\deg(z)=0$ are treated similarly.
Let $z=\sum\limits_{i=1}^r \alpha_i \left(\sum\limits_{p\in Arr(C_i)} p(z(C_i)^*)^{d/n_i}p^*\right) +z' ,\, z'\in I(W),\,\deg(z)=d\geq 1.$
Let $v\in V(C_i).$ There exists a sufficiently long closed path $p,\, s(p)=r(p)=v,$ such that $z' p=0.$ Indeed, let $z'=\sum \beta_j p_jq_j^*,\, \beta_j\in F; p_j,q_j\in Path(\Gamma),
s(p_j)=r(q_j)\in V\setminus W,$ $r(p_j)=r(q_j)\in W.$  If $l(p)> \max\limits_{j} l(q_j)$ then each $q_j$ is not a beginning of the path $p,$ hence $q_j^*p=0.$
Let $p'$ be a path such that $s(p')=v,\, r(p')\in W.$ Then $pp'z=0, zpp'=\alpha_i z(C_i)^{d/n_i} pp'.$ It is easy to see that  $z(C_i)^{d/n_i} pp'$ is a nonzero path from $v$ to $r(p'),$ which contradicts centrality of $z$ and finishes the proof of the Lemma.
\end{proof}

\begin{proof}[Proof of Theorem 3.]
Let $\deg(z)=0,$ $z=\sum\limits_{i=1}^r \alpha_i \left(\sum\limits_{p\in Arr(U_i)} pp^*\right)+z',$ $0\neq  \alpha_i \in F; U_1,\cdots, U_r$ are disjoint finitary hereditary subsets in $\Gamma/ W,$ $z'\in I(W).$

For each $U_i$ let $U'_i$ be the set of all  sinks and cyclic vertices lying in $U_i$ and their descendants. By Lemma 19 each set $U'_i$ is hereditary in $\Gamma.$ It is easy to see that $U'_i$ is also finitary.
The element $z_1=\sum\limits_{i=1}^r \alpha_i e(U'_i)$ lies in the center $Z(L(\Gamma)).$ Since $z_1\in \sum\limits_{v\in V\setminus W} vL(\Gamma)v$ it follows that $z'W=(0).$
Now, $z-z_1\in I(W), (z-z_1)^2=0.$ Hence $z=z_1.$
If $\deg(z)=d\geq 1$ then $z=\sum\limits_{j}\beta_j \left(\sum\limits_{p\in Arr(C_i)} pz(C_i)^{d/n_i}p^*\right)+z'',$ $0\neq \beta_j\in F, C_j$ are disjoint NE-cycles in $\Gamma/W$ of length $n_j,$ $z''\in I(W).$
By Lemma 19 each $C_j$ does not have exits in $\Gamma.$ Hence the element $z_2= \sum\limits_{j} \beta_j \left(\sum\limits_{p\in Arr(C_j)} pz(C_j)^{d/n_j} p^*\right)$ lies in $Z(L(\Gamma)).$

Arguing as above we conclude that $z=z_2.$ The case $\deg(z)=-d, d\geq 1,$ is treated similarly. This finishes the proof of the Theorem.
\end{proof}

\section*{Acknowledgement}
This project was funded by the Deanship of Scientific Research (DSR), King Abdulaziz University, under Grant No.
(27-130-36-HiCi). The authors, therefore, acknowledge technical and financial support of KAU.

\end{document}